\documentclass[12pt, reqno]{amsart}

\usepackage{amsmath, amsthm, amssymb}
\usepackage{enumitem}
\usepackage{pdflscape}
\usepackage{caption}

\usepackage{nccmath}

\usepackage{bm}

\usepackage{ifpdf}
\ifpdf
\usepackage[pdftex]{graphicx}
\else
\usepackage[dvips]{graphicx}
\fi
\usepackage{tikz}
 	 \usetikzlibrary{arrows,backgrounds}
\usepackage[all]{xy}

\usepackage{xifthen}

\usepackage{multicol}

\usepackage{tocvsec2}

\usepackage{bbm}

\usepackage{MnSymbol}

\input xy
\xyoption{all}

\usepackage[pdftex,plainpages=false,hypertexnames=false,pdfpagelabels]{hyperref}
\newcommand{\arxiv}[1]{\href{http://arxiv.org/abs/#1}{\tt arXiv:\nolinkurl{#1}}}
\newcommand{\arXiv}[1]{\href{http://arxiv.org/abs/#1}{\tt arXiv:\nolinkurl{#1}}}

\newcommand{\googlebooks}[1]{(preview at \href{http://books.google.com/books?id=#1}{google books})}

\usepackage{xcolor}
\definecolor{dark-red}{rgb}{0.7,0.25,0.25}
\definecolor{dark-blue}{rgb}{0.15,0.15,0.55}
\definecolor{medium-blue}{rgb}{0,0,.8}
\definecolor{DarkGreen}{RGB}{0,150,0}
\definecolor{rho}{named}{red}
\hypersetup{
   colorlinks, linkcolor={purple},
   citecolor={medium-blue}, urlcolor={medium-blue}
}

\usepackage{longtable}
\usepackage{fullpage}

\theoremstyle{plain}
\newtheorem{thm}{Theorem}[section]
\newtheorem*{thm*}{Theorem}

\newtheorem*{cor*}{Corollary}

\newtheorem*{conj*}{Conjecture}
\newtheorem{lem}[thm]{Lemma}

\newtheorem{prop}[thm]{Proposition}

\newtheorem*{quest*}{Question}
\newtheorem*{claim*}{Claim}

\theoremstyle{definition}
\newtheorem{defn}[thm]{Definition}

\newtheorem{alg}[thm]{Algorithm}

\newtheorem{sub-ex}[thm]{Sub-Example}
\newtheorem{rem}[thm]{Remark}
\newtheorem*{rem*}{Remark}

\usepackage{graphicx}
\usepackage{relsize}

\DeclareMathOperator{\Aut}{Aut}

\DeclareMathOperator{\coev}{coev}

\DeclareMathOperator{\ev}{ev}

\DeclareMathOperator{\Hom}{Hom}

\DeclareMathOperator{\id}{id}

\newcommand{\comment}[1]{}

\newcommand{\be}{\begin{enumerate}[label=(\arabic*)]}
\newcommand{\ee}{\end{enumerate}}

\newcommand{\C}{\mathbb{C}}

\newcommand{\cC}{\mathcal{C}}
\newcommand{\cD}{\mathcal{D}}
\newcommand{\cR}{\mathcal{R}}
\newcommand{\cY}{\mathcal{Y}}

\usetikzlibrary{calc}
\newcommand*{\vertchar}[2][0pt]{%
  \tikz[
    inner sep=0pt,
    shorten >=-.15ex,
    shorten <=-.15ex,
    line cap=round,
    baseline=(c.base),
  ]\draw
    (0,0) node (c) {#2}
    ($(c.south)+(#1,0)$) -- ($(c.north)+(#1,0)$);%
}

\def\semicolon{;}
\def\applytolist#1{
    \expandafter\def\csname multi#1\endcsname##1{
        \def\multiack{##1}\ifx\multiack\semicolon
            \def\next{\relax}
        \else
            \csname #1\endcsname{##1}
            \def\next{\csname multi#1\endcsname}
        \fi
        \next}
    \csname multi#1\endcsname}

\def\calc#1{\expandafter\def\csname c#1\endcsname{{\mathcal #1}}}
\applytolist{calc}QWERTYUIOPLKJHGFDSAZXCVBNM;
\def\bbc#1{\expandafter\def\csname bb#1\endcsname{{\mathbb #1}}}
\applytolist{bbc}QWERTYUIOPLKJHGFDSAZXCVBNM;
\def\bfc#1{\expandafter\def\csname bf#1\endcsname{{\mathbf #1}}}
\applytolist{bfc}QWERTYUIOPLKJHGFDSAZXCVBNM;
\def\sfc#1{\expandafter\def\csname s#1\endcsname{{\sf #1}}}
\applytolist{sfc}QWERTYUIOPLKJHGFDSAZXCVBNM;
\def\fc#1{\expandafter\def\csname f#1\endcsname{{\mathfrak #1}}}
\applytolist{fc}QWERTYUIOPLKJHGFDSAZXCVBNM;

\renewcommand{\Vec}{{\sf Vec}}

\newcommand{\noshow}[1]{}
\renewcommand{\MR}[1]{}

\makeatletter
\newcommand{\hashdef}[2]{\@namedef{#1}{#2}}
\newcommand{\hashlookup}[1]{\@nameuse{#1}}
\makeatother

\usepackage{longtable}

\usetikzlibrary{shapes}
\usetikzlibrary{cd}
\usetikzlibrary{decorations,decorations.pathreplacing,decorations.markings}
\usetikzlibrary{fit,calc,through}
\usetikzlibrary{external}
\usetikzlibrary{arrows,backgrounds,patterns.meta}
\tikzset{vertex/.style = {shape=circle,draw,fill=black,inner sep=0pt,minimum size=5pt}}
\tikzset{edge/.style = {->,> = latex', bend right}}
\tikzset{
	super thick/.style={line width=3pt}
}
\tikzset{
    quadruple/.style args={[#1] in [#2] in [#3] in [#4]}{
        #1,preaction={preaction={preaction={draw,#4},draw,#3}, draw,#2}
    }
}
\tikzstyle{shaded}=[fill=red!10!blue!20!gray!30!white]
\tikzstyle{unshaded}=[fill=white]
\tikzstyle{empty box}=[circle, draw, thick, fill=white, opaque, inner sep=2mm]
\tikzstyle{annular}=[scale=.7, inner sep=1mm, baseline]
\tikzstyle{rectangular}=[scale=.75, inner sep=1mm, baseline=-.1cm]
\tikzstyle{mid>}=[decoration={markings, mark=at position 0.5 with {\arrow{>}}}, postaction={decorate}]
\tikzstyle{mid<}=[decoration={markings, mark=at position 0.5 with {\arrow{<}}}, postaction={decorate}]
\tikzstyle{over}=[double, draw=white, super thick, double=]

\newcommand{\roundNbox}[6]{
	\draw[rounded corners=5pt, very thick, #1] ($#2+(-#3,-#3)+(-#4,0)$) rectangle ($#2+(#3,#3)+(#5,0)$);
	\coordinate (ZZa) at ($#2+(-#4,0)$);
	\coordinate (ZZb) at ($#2+(#5,0)$);
	\node at ($1/2*(ZZa)+1/2*(ZZb)$) {#6};
}

  \newcommand{\tikzmath}[2][]
     {\vcenter{\hbox{\begin{tikzpicture}[#1]#2
                     \end{tikzpicture}}}
     }

\tikzdeclarepattern{
  name=primeddots,
  type=uncolored,
  bounding box={(-.6pt,-.6pt) and (.6pt,.6pt)},
  tile size={(3pt,3pt)},
  tile transformation={rotate=60},
  parameters={none}, 
  code={
    \fill(0pt,0pt) circle (.35pt);
  }
}

\tikzstyle{primedregion}[none]=[
	preaction={fill=#1},
	pattern=primeddots,
  draw=#1,
]

\definecolor{violet}{RGB}{148,0,211}

\newcommand{\maxwidth}[2][\linewidth]{
  \setbox0=\hbox{#2}
  \ifthenelse{\dimtest{\wd0}>{#1}}%
    {\resizebox{#1}{!}{#2}}
    {\usebox0}
}

\tikzset{->-/.style={decoration={
  markings,
  mark=at position .5 with {\arrow{>}}},postaction={decorate}}}
\tikzset{-<-/.style={decoration={
  markings,
  mark=at position .5 with {\arrow{<}}},postaction={decorate}}}

\tikzset{-->-/.style={decoration={
  markings,
  mark=at position .7 with {\arrow{>}}},postaction={decorate}}}

\usetikzlibrary{arrows,matrix}
\usepackage{mathtools}
\newcommand{\oVec}{\operatorname{Vec}(G,\omega, \pi)}

\newcommand{\oVecc}{\operatorname{Vec}(G,\omega)}

\usepackage{url}

\begin{document}

\title{A diagrammatic presentation for every pivotal pointed fusion category}
\author{Chumeng Di}
\address{Department of Mathematics, The Ohio State University, 100 Math Tower, 231 W 18th Ave, Columbus, OH 43210}
\email{di.59@osu.edu}

\author{Anup Poudel}
\address{School of Mathematics, Georgia Institute of Technology, 686 Cherry St NW, Atlanta, GA 30332}
\email{apoudel33@gatech.edu}

\maketitle

\begin{abstract}
   We provide a generators and relations presentation of pivotal pointed fusion categories, $\oVec$. Unlike the well-known skeletal model, our presentation is strict and allows multiple isomorphic objects. Our main tool is skein theory, which allows us to apply topological tools to understand the relations of morphisms in the category. This expository paper provides details of skein-theoretic work in the setup that allows multiple isomorphic labels in the generators, which didn't previously exist in the literature.
\end{abstract}

\section{Introduction}
\typeout{
introduction 1-2 pages
why do we care about this problem? Why tensor cat? Significance of Vec G omega? Why care about planar presentations of categories? Why graphical calculus?}

Tensor categories arise naturally in quantum physics, as they describe symmetries of quantum systems and are central objects of study in quantum algebra and quantum topology. Traditionally, transformations that describe the symmetries of a physical system form a group. Thus we use group actions on a system to characterize its symmetries. If the transformations which can be consecutively applied to a system are no longer required to be invertible, then monoid actions should be considered instead. When our system of interest is an object of an $n$-category, its higher symmetries are characterized by the action of a higher monoid, i.e., a monoidal $(n-1)$-category.  Important examples of quantum systems such as $C^*$-algebras, von Neumann algebras, and quantum spin chains are objects of 2-categories, and their symmetries form tensor categories.

Diagrammatics of tensor categories are 2-dimensional in nature: morphisms compose both vertically under composition and horizontally under tensor product. Thus, to write down a presentation for a tensor category, we use 2-dimensional graphical calculus to denote generating objects, generating morphisms, and relations. Such a skein-theoretic approach makes the algebraic structure visible and allows one to do calculations in a tensor category. 

Skein theory is an important tool in quantum topology that has led to many new connections among the fields of quantum algebra, low dimensional topology and subfactor theory. Skein theoretic presentations of pivotal tensor categories associated to simple Lie algebras of rank 2 were given by Kuperberg in \cite{Kup}. Since then, many other families of pivotal tensor categories have been presented skein-theoretically, see for example \cite{CKM, BERT, TVW}. These diagrammatic presentations have been crucial for categorification of quantum invariants as the skein relations naturally lead to relations between ``foams" at the $2-$categorical level which have given rise to new homology theories, see for example \cite{Kh, KL}. In the context of subfactor theory and planar algebras, Morrison-Peters-Snyder gave a skein-theoretic presentation of the dihedral planar algebra \cite{MPS10} and Bigelow \cite{Big10} gave a skein-theoretic presentation of ADE planar algebras. Both works inspired the construction of a new subfactor planar algebra in \cite{BMPS12} and completed the classification of a class of subfactors. There has been much work surrounding diagrammatic presentations of planar algebras and fusion categories such as \cite{Pet09}, \cite{MP15}, \cite{czenky2024Zn} and \cite{Mol24}.

Pivotal categories naturally arise in TQFTs. The pivotality structure that coherently assigns a dual $x^\vee$ to each $x$ such that $x \simeq x^{\vee \vee}$ comes from the orientation reversal duality on manifolds in the cobordism category. In this paper, we study pivotal pointed fusion categories $\oVec$ which naturally arise in a particular example of TQFT called the Dijkgraaf-Witten theory. A pivotal presentation of $\oVec$ will be useful in characterizing its actions.

It is well-accepted that for a general Cauchy complete tensor category $\cC,$ we can ask for at most one of the properties skeletal and strict. As one of the first examples of tensor categories, the skeletal $\oVecc$ has been well-studied \cite{EGNO}. We provide a skein-theoretic\footnote{We follow an optimistic convention for our diagrams i.e., we read them from bottom to top. As usual, the 2D diagrammatic calculus for fusion categories suppresses associators and unitors.} presentation of the strict pivotal $\oVec$. Our presentation includes the entire multiplication table of $G$ as generating morphisms. As a possible future project, we seek to reduce the number of generators in our presentation for $\oVec$, when a group presentation for $G$ is given.

\begin{thm}\label{main theorem intro}
    The following is a presentation for pivotal category $\oVec$:
    
\vspace{.3cm}
\textnormal{Generating objects:}
\begin{enumerate}[label=\textup{(O\arabic*)}]
\item 
For each $g \in G,$ $g \neq \mathbf{1},$ we have objects $g:\
    \tikzmath{
    \draw[thick, blue, ->](0,0) -- (0,.6);
    \draw[thick, blue](0,.6) --(0,.8);
    \node[blue] at (.3, 0.3) {$g$};
    }$
    and
    $g^*:\
    \tikzmath{
    \draw[thick, blue, ->](0,.8) -- (0,.5);
    \draw[thick, blue](0,.5) --(0,0);
    \node[blue] at (.3, 0.3) {$g^*$};
    }$. 
\end{enumerate}   

\vspace{.3cm}
\textnormal{Generating isomorphisms:}
\begin{enumerate}[label=\textup{(M\arabic*)}]
\item 
For each pair $g,h \in G$, choose an isomorphism
$\mu_{g,h}: gh \rightarrow g \otimes h$
$
\tikzmath{
\draw[thick, blue,-->-] (0,0) -- (-.4,.4) node[above]{$\scriptstyle g$};
\draw[thick,blue, -->-] (0,0) -- (.4,.4) node[above]{$\scriptstyle h$};
\draw[thick, blue, ->-] (0,-.4) node[below]{$\scriptstyle gh$} -- (0,0);
\filldraw[blue] (0,0) node[left]{$\scriptstyle \mu_{g,h}$} circle (.0cm);
}$ to be the inverse of the tensorator of $F$, and the tensorator denoted by
$\tikzmath{
    \draw[thick, blue, -->-] (0,0) -- node[above,yshift=.1cm] {$\scriptstyle gh$}(0,0.4);
    \draw[thick, blue, ->] (-0.4,-0.4) node[below] {$\scriptstyle g$} -- (-0.2,-0.2);
    \draw[thick, blue] (-0.2, -0.2) -- (0,0);
    \draw[thick, blue, ->] (0.4,-0.4) node[below, xshift=.1cm] {$\scriptstyle h$} -- (0.2,-0.2);
    \draw[thick, blue] (0.2, -0.2) -- (0,0);
     \node[blue] at (.5,0.1) {$\scriptstyle \mu_{g,h}^{-1}$};
    }$.
    When $h=g^{-1}$, we denote the tensorator by
    $\tikzmath{
\draw[thick, blue, ->-] (-.4,-.4) -- (0,0);
\draw[thick,blue,->-] (.4,-.4) -- (0,0); 
\node[blue] at (-0.6, -.3) {$\scriptstyle g$};
\node[blue] at (.7,-.3) {$\scriptstyle g^{-1}$};
}$.

\item 
For each pair of duals $g, g^*$, $g \neq \mathbf{1},$ we have isomorphisms $\coev_{g^*} $
$\tikzmath{
\draw[thick, blue, ->] (0,0) arc(-180:0:0.3);
\node[blue] at (-0.2, 0) {$\scriptstyle g^*$};
\node[blue] at (0.8, 0) {$\scriptstyle g$};
}$,
$\ev_{g^*}$
$\tikzmath{
\draw[thick, blue, ->] (0,0) arc(180:0:0.3);
\node[blue] at (-0.2, 0) {$\scriptstyle g$};
\node[blue] at (0.9, 0) {$\scriptstyle g^*$};
}$,
$\coev_{g}$
$\tikzmath{
\draw[thick, blue, <-] (0,0) arc(-180:0:0.3);
\node[blue] at (-0.2, 0) {$\scriptstyle g$};
\node[blue] at (0.8, 0) {$\scriptstyle g^*$};
}$,
$\ev_{g}$
$\tikzmath{
\draw[thick, blue, <-] (0,0) arc(180:0:0.3);
\node[blue] at (-0.25, 0) {$\scriptstyle g^*$};
\node[blue] at (0.9, 0) {$\scriptstyle g$};
}$
that witness duality.
\end{enumerate}   

\vspace{.3cm}
\textnormal{Relations:}
\begin{enumerate}[label=\textup{(R\arabic*)}]
\item
For each $(g,h,k) \in G^3$, 
$\tikzmath{
\draw[thick, blue] (.3,.3) node[above]{$\scriptstyle k$} -- (0,0) -- (-.3,.3) node[above]{$\scriptstyle h$};
\draw[thick, blue] (0,0) -- (-.3,-.3) -- (-.9,.3) node[above]{$\scriptstyle g$};
\draw[thick, blue] (-.3,-.3) -- (-.3,-.6) node[below]{$\scriptstyle ghk$};
\filldraw[blue] (0,0) node[right]{$\scriptstyle \mu_{h,k}$} circle (.0cm);
\filldraw[blue] (-.3,-.3) node[right]{$\scriptstyle \mu_{g,hk}$} circle (.0cm);
}
=
\omega(g,h,k)
\tikzmath{
\draw[thick, blue] (-.3,.3) node[above]{$\scriptstyle g$} -- (0,0) -- (.3,.3) node[above]{$\scriptstyle h$};
\draw[thick, blue] (0,0) -- (.3,-.3) -- (.9,.3) node[above]{$\scriptstyle k$};
\draw[thick, blue] (.3,-.3) -- (.3,-.6) node[below]{$\scriptstyle ghk$};
\filldraw[blue] (0,0) node[left]{$\scriptstyle \mu_{g,h}$} circle (.0cm);
\filldraw[blue] (.3,-.3) node[left]{$\scriptstyle \mu_{gh,k}$} circle (.0cm);
}
$, for some $\ \omega \in [\omega]$.

\item For each $g \in G,\ g \neq \mathbf{1}:$
 $\tikzmath{
\draw[thick, blue, ->] (0,0) -- (0,0.4);
\draw[thick, blue, ->] (-0.4, -0.4) -- (-0.2, -0.2);
\draw[thick, blue] (-.2,-.2) -- (0,0);
\draw[thick, blue, dotted] (0.4,-0.4) -- (0,0);
\node[blue] at (-0.2, 0.2) {$\scriptstyle g$};
}$
=
$\tikzmath{
\draw[thick, blue, ->](0,-0.4) -- (0,0.4);
\node[blue] at (0.2,0) {$\scriptstyle g$};
}$ and  
$\tikzmath{
\draw[thick, blue, ->] (0,0) -- (0,0.4);
\draw[thick, blue, ->] (0.4, -0.4) -- (0.2, -0.2);
\draw[thick, blue] (.2,-.2) -- (0,0);
\draw[thick, blue, dotted] (-0.4,-0.4) -- (0,0);
\node[blue] at (-0.2, 0.2) {$\scriptstyle g$};
}$
=
$\tikzmath{
\draw[thick, blue, ->](0,-0.4) -- (0,0.4);
\node[blue] at (0.2,0) {$\scriptstyle g$};
}.$

\vspace{.3cm}

\item For each $g \in G,\ g \neq \mathbf{1}:$
$\tikzmath{
\draw[thick, blue] (-.6,-.5) -- node[left]{$\scriptstyle g^*$} (-.6,0) arc(180:0:.3) ;
\draw[thick, blue, <-] (0,0) node[right] {$\scriptstyle g$} arc(-180:0:.3) -- node[right]{$\scriptstyle g^*$} (.6,.5);
}
=
\tikzmath{
\draw[thick, blue, ->] (0,.5) -- node[right]{$\scriptstyle g^*$} (0,0);
\draw[thick, blue] (0,0) -- (0,-.5);
},$
$\tikzmath{
\draw[thick, blue] (-.6,.5) -- node[left]{$\scriptstyle g$} (-.6,0) arc(-180:0:.3) ;
\draw[thick, blue, <-] (0,0) node[right] {$\scriptstyle g^*$} arc(180:0:.3) -- node[right]{$\scriptstyle g$} (.6,-.5);
}
=
\tikzmath{
\draw[thick, blue] (0,.5) -- node[right]{$\scriptstyle g$} (0,0);
\draw[thick, blue, <-] (0,0) -- (0,-.5);
},$
$\tikzmath{
\draw[thick, blue,->] (-.6,-.5) -- node[left]{$\scriptstyle g$} (-.6,0) arc(180:0:.3) ;
\draw[thick, blue] (0,0) node[right] {$\scriptstyle g^*$} arc(-180:0:.3) -- node[right]{$\scriptstyle g$} (.6,.5);
}=
\tikzmath{
\draw[thick, blue] (0,.5) -- node[right]{$\scriptstyle g$} (0,0);
\draw[thick, blue, <-] (0,0) -- (0,-.5);
},$
$\tikzmath{
\draw[thick, blue,->] (-.6,.5) -- node[left]{$\scriptstyle g^*$} (-.6,0) arc(-180:0:.3) ;
\draw[thick, blue] (0,0) node[right] {$\scriptstyle g$} arc(180:0:.3) -- node[right]{$\scriptstyle g^*$} (.6,-.5);
}
=
\tikzmath{
\draw[thick, blue,->] (0,.5) -- node[right]{$\scriptstyle g^*$} (0,0);
\draw[thick, blue] (0,0) -- (0,-.5);
}.$

\vspace{.3cm}

\item 
For each $g \in G,\ g \neq \mathbf{1}:$
$\tikzmath{
\draw[thick, blue, ->] (-.3,0) arc (180:-180:.3cm);
\node[blue] at (-.5,0) {$\scriptstyle g$};
\node[blue] at (.5,0) {$\scriptstyle g^*$};
}
=
\pi(g)
$.
\end{enumerate}  
\end{thm}


\subsection{Acknowledgement}

The authors would like to thank David Penneys for suggesting this problem and many helpful conversations. C.D. was partially supported by NSF DMS-2154389 and NSF DMS-2554723 and A.P. acknowledges support from an AMS-Simons Travel Grant.

\section{\texorpdfstring{$\oVec$}{oVec} as a skeletal pivotal category}\label{sec:1}

\typeout{
Fix conventions. Associator, trivial unitor of monoidal category; give name to monoidal functor data

$\oVec$ as a monoidal category, well-defined (cohomologous implies equivalent) does $G$ have to be finite group? Ever used semisimplicity somewhere?

$\oVec$ admits a dual functor, dual functor is a property.

Definition of pivotal structure and spherical structure; $\oVec$ admits a spherical structure; the pivotal structures on $\oVec$ are classified by..., where a $\pi \in \Hom(G, \bbC^\times)$ represents the pivotal structure given by... briefly mention $\Aut_\otimes(\id_\cC)$ .. say in the following notation $\oVec$ represents $(\oVec, \pi)$}



 For details regarding the definitions of monoidal categories we refer the reader to \cite{EGNO}. In a linear monoidal category $\cC$, we denote the associator isomorphisms by $\alpha_{a,b,c}: a \otimes (b \otimes c) \to (a \otimes b) \otimes c$ for $a,b,c \in \cC$, unitor isomorphisms by $\lambda_a: 1_\cC \otimes a \to a$ and $\rho_a: a \otimes 1_\cC  \to a$ for $a \in \cC$. For a monoidal functor $F: \cC_1 \to \cC_2$ between monoidal categories, we denote the tensorator isomorphisms by $F^2_{a,b}: F(a) \otimes F(b) \to F(a \otimes b)$.\\

For a finite group $G$ and a 3-cocycle $\omega \in Z^3(G ; U(1))$, $\oVecc$ is the monoidal category where objects are $G$-graded vector spaces, i.e., vector spaces in the form $V = \bigoplus_{g \in G} V_g$, and morphisms are linear maps which preserve the grading. Tensor product is given by
\[
(V \otimes W)_g \coloneq \bigoplus_{hk =g} V_h \otimes W_k.
\]
Associators and unitors are given by 
\[\alpha_{g,h,k} \coloneq \omega(g,h,k) \id_{ghk},\ \lambda_g \coloneq \omega(\mathbf{1},\mathbf{1},g) \id_g,\  \rho_g \coloneq  \omega(g,\mathbf{1},\mathbf{1})^{-1} \id_g,\]
where group elements $g \coloneq \C_g$ denote simple objects of this category. There are $|G|$ simples in our skeletal model $\oVecc$, one in each equivalence class. The 3-cocycle condition of $\omega$ is exactly what makes the associators defined above satisfy the pentagon relation in a monoidal category:

\[\begin{tikzcd}
	{g(h(kl))} && {(gh)(kl)} \\
	&&& {((gh)k)l} \\
	{g((hk)l)} && {(g(hk))l}
	\arrow["{\omega(g,h,kl)}", from=1-1, to=1-3]
	\arrow["{\omega(h,k,l)}", from=1-1, to=3-1]
	\arrow["{\omega(gh,k,l)}"{pos=0.4}, from=1-3, to=2-4]
	\arrow["{\omega(g,hk,l)}", from=3-1, to=3-3]
	\arrow["{\omega(g,h,k)}"'{pos=0.4}, from=3-3, to=2-4]
\end{tikzcd}\]

Cohomologous $\omega$ and $\omega'$ define monoidally equivalent categories $\oVecc$ and $\operatorname{Vec}(G,\omega')$. Let $\omega' = \omega \cdot \partial^3(\mu)$ for some $\mu: G^2 \to U(1).$ Define $F: \oVecc \to \operatorname{Vec}(G,\omega')$ trivially on objects and morphisms. For $F^2_{g,h} \coloneq \mu(g,h) \id_{gh}$, $\omega' = \omega \cdot \partial^3(\mu)$ is exactly the monoidality axiom of $F,$ and $F$ is a monoidal equivalence. \comment{We may write $\Vec_G^{[\omega]}$ for this monoidal category.}\\

A choice of dual for each simple $g$, $(g^\vee \coloneq g^{-1}$,\ $\ev_g: g^{-1} \otimes g = \mathbf{1} \to \mathbf{1},\ \coev_g:\mathbf{1} \to g \otimes g^{-1}=\mathbf{1})$ assemble into a dual functor $()^\vee: \oVecc \to \oVecc^{\text{mop}}$, where $\cC^\text{mop}$ denotes the category obtained by reversing the arrows and monoidal product of $\cC$. Dual functors on a tensor category are unique up to unique monoidal natural isomorphism.\\

We include the definitions concerning pivotal categories that we will need; for more detailed background, we refer the reader to \cite{penneys2018unitarydualfunctorsunitary}.

\begin{defn} \label{pivotal structure}
    A \textbf{pivotal structure} on a monoidal category $\cC$ is a choice of dual functor $()^\vee$ together with a monoidal natural isomorphism $\varphi: \id_\cC \Rightarrow ()^{\vee \vee}.$ Naturality of $\varphi$ is a rotation by $2\pi$ in graphical calculus: 
    \[
    \tikzmath{
    \draw[thick,blue] (0,0) -- node[left]{$\scriptstyle a$} (0,.4);
    \roundNbox{blue}{(0, .7)}{.3}{0}{0}{$\varphi_a$};
    \draw[thick, blue] (0,1) -- node[left]{$\scriptstyle a^{\vee \vee}$} (0,2) arc(180:0:.8) -- (1.6,1.2) arc(0:-180:.3)  node[left,yshift =-.2cm]{$\scriptstyle a$};
    \roundNbox{blue}{(1,1.5)}{.3}{0}{0}{$f$};
    \draw[thick,blue] (1,1.8) node[right, yshift =.2cm]{$\scriptstyle b$} arc(0:180:.3) -- (.4,1.2) arc(-180:0:.8) -- (2,3) node[right]{$\scriptstyle b^{\vee \vee}$};
    }=
    \tikzmath{
    \draw[thick,blue] (0,0) -- node[right]{$\scriptstyle a$} (0,.4);
    \roundNbox{blue}{(0,.7)}{.3}{0}{0}{$f$};
    \draw[thick,blue] (0,1) -- (0,1.4) node[right, yshift = -.2cm]{$\scriptstyle b$};
    \roundNbox{blue}{(0,1.7)}{.3}{0}{0}{$\varphi_b$};
    \draw[thick,blue] (0,2) -- (0,2.4) node[right]{$\scriptstyle b^{\vee \vee}$};
    }
    \hspace{1.5cm}
    \forall f \in \Hom(a,b).
    \]
\end{defn}
\begin{defn}
     Let $(\cC, \varphi: \id_{\cC} \Rightarrow ()^{\vee \vee})$ be a pivotal category. For any object $c \in \cC,$ we have
\begin{align*}
    &\mathbf{1} \xrightarrow{\coev_{c^\vee}} c^\vee \otimes c^{\vee \vee} \xrightarrow{ \id_{c^\vee} \otimes \varphi_c^{-1}} c^{\vee} \otimes c\xrightarrow{\ev_{c}} \mathbf{1} = d^L_c \cdot \id_\mathbf{1}
    \\
  &\mathbf{1} \xrightarrow{\coev_c} c \otimes c^\vee \xrightarrow{\varphi_c \otimes \id_{c^\vee}} c^{\vee \vee} \otimes c^\vee \xrightarrow{\ev_{c^\vee}} \mathbf{1} = d^R_c \cdot \id_\mathbf{1}
\end{align*}
for some scalars $d^L_c, d^R_c \in \C.$ Then $\lambda_c, \rho_c$ are the \textbf{left/right quantum dimensions} of $c,$ respectively, denoted $\dim_L^\varphi(c) \coloneq d^L_c,$ and $\dim_R^\varphi(c) \coloneq d^R_c.$
\end{defn}

 The monoidal category $\oVecc$ admits a unique pseudounitary\footnote{It is indeed unitary.} (quantum dimensions of objects are strictly positive) pivotal structure $\varphi_0: \id_{\oVecc} \to ()^{\vee \vee}$. In this pseudounitary category, the left/right quantum dimensions of a simple object $g$ are 
 \[
 \dim_L^{\varphi_0}(g) = \dim_R^{\varphi_0}(g) =1
 \]
 for all $g \in G.$
\comment{ 
 \[\tikzmath{
\draw[thick, blue] (0, .3) arc (180:0:.3)  --node[right] {$\scriptstyle g^{-1}$} (0.6, -0.3) arc(0:-180:0.3) node[left, yshift = -0.15cm] {$\scriptstyle g$};
\roundNbox{blue}{(0,0)}{.3}{0}{0}{$\varphi_g$};
}
=
\tikzmath{
\draw[thick, blue] (0, .3) arc (0:180:.3)  --node[left] {$\scriptstyle g^{-1}$} (-0.6, -0.3) arc(-180:0:0.3) node[right, yshift = -0.15cm] {$\scriptstyle g$};
\roundNbox{blue}{(0,0)}{.3}{0.05}{0.05}{$\varphi_g^{-1}$};
}
=1\]} \\

In general, for a tensor category $\cC$ with a chosen pivotal structure $\varphi: \id_{\cC} \to ()^{\vee \vee}$, the collection of all pivotal structures on $\cC$ are obtained by composition with monoidal natural transformations in $\Aut_\otimes(\id_{\cC})$: for each $\delta \in \Aut_\otimes(\id_{\cC}),$ $\varphi \circ \delta$ defines a pivotal structure. One can check that the data of a monoidal natural isomorphism $\delta$ on $\id_\cC$ is exactly the data of a group homomorphism $\delta : U_\cC \to \C^\times $ from the universal grading group $U_\cC$ to $\C^\times$ (see \cite[\S 4.14]{EGNO}). For $\cC=\oVecc,$ $U_\cC = G.$ If we let the pseudounitary structure $\varphi_0$ be represented by the trivial map in $\Hom(G \to \C^\times)$, then any group homomorphism $\pi \in \Hom(G \to \C^\times)$ specifies a pivotal structure on $\oVecc$. In pivotal category $\oVec$, the left/right quantum dimensions of a simple object $g$ are 
\begin{equation}
\dim_L^{\pi}(g) = \pi(g)^{-1},\ \dim_R^{\pi}(g) = \pi(g)
\label{eq: quantum dims of simple}
 \end{equation}
for all $g \in G.$ 
\comment{\[\tikzmath{
\draw[thick, blue] (0,0) -- (0,0.2);
\roundNbox{blue}{(0,0.5)}{0.3}{0}{0}{$\varphi_g$};
\draw[thick, blue] (0,0.8) arc(180:0:0.4) -- node[right] {$\scriptstyle g^{-1}$} (0.8, -0.6) ;
\draw[thick, blue] (0.8, -0.6) arc(0:-180:0.4);
\roundNbox{blue}{(0,-0.3)}{0.3}{0.2}{0.2}{$\pi(g)$}
} = \pi(g)\
\text{and}\
\tikzmath{
\draw[thick, blue] (0,0) -- (0,.2);
\roundNbox{blue}{(0,0.5)}{0.3}{.3}{.25}{$\pi(g^{-1})$};
\draw[thick, blue] (0,0.8) arc(0:180:0.4) -- node[left] {$\scriptstyle g^{-1}$} (-0.8, -0.6);
\draw[thick, blue] (-0.8,-0.6) arc(-180:0:0.4);
\roundNbox{blue}{(0, -0.3)}{0.3}{0.1}{0.1}{$\varphi_g^{-1}$};
}
= \pi(g^{-1})\]}

For pivotal category $\oVec$, there is a gauge freedom in choosing the duality data: one may rescale $\ev_g$, $\coev_g,$ and thus $\varphi_g$ all at the same time in defining the same pivotal category. For the convenience of later sections, we fix our choices 
\begin{equation}
\label{eq:duality-choice}
\operatorname{ev}_g=\operatorname{id}_{\mathbf 1},
\qquad
\operatorname{coev}_g=\omega(g,g^{-1},g)\operatorname{id}_{\mathbf 1},
\qquad
\varphi_g=\omega(g^{-1},g,g^{-1})\pi(g)\operatorname{id}_g,
\quad g\in G.
\end{equation}

\begin{defn} \label{pivfunctor}
    For a monoidal functor $(F,F^2): (\cC,\varphi^\cC) \to (D, \varphi^\cD)$ between pivotal categories, there is a canonical natural isomorphism $\delta_c: F(c^\vee) \to F(c)^\vee$ for each $c \in \cC$ given by
    \[\delta_c \coloneq
    \tikzmath{
    \roundNbox{blue}{(0,0)}{0.3}{.4}{.4}{$F^2_{c^\vee,c}$};
    \draw[thick,blue] (-.05,.3) -- (-.05,.9);
    \draw[thick,blue] (.05,.3) -- (.05,.9);
    \roundNbox{blue}{(0,1.2)}{0.3}{.4}{.4}{$F(\ev_c)$};
    \draw[thick,blue] (-.2,-.3) -- (-.2,-1.1);
    \draw[thick,blue] (.2,-.3) arc(-180:0:.6) -- (1.4,1.7);
    \node[blue] at (-.7,-.8) {$\scriptstyle F(c^\vee)$};
    \node[blue] at (-.7,.6) {$\scriptstyle F(c^\vee \otimes c)$};
    \node[blue] at (1.9,1.5) {$\scriptstyle F(c)^\vee$};
    }.\]
    We call $(F, F^2)$ a \textbf{pivotal functor} if 
    \[
    \delta_c^\vee \circ\varphi^\cD_{F(c)} = \delta_{c^\vee} \circ F(\varphi^\cC_c).
    \]
\end{defn}

\section{strictified $\oVec$}\label{sec: 3}

\begin{defn}
    A \textbf{strict pivotal category} is a pivotal strict monoidal category in which the tensorator of the dual functor $(()^\vee)^2_{a,b}: a^\vee \otimes_{\text{mop}} b^\vee \to (a \otimes b)^\vee$ and the pivotal structure $\varphi: \id \to ()^{\vee \vee}$ are both identities.
\end{defn}

\begin{prop} \cite[Thm 2.2.]{Ng_2007} 
Every pivotal category is equivalent, as a pivotal category, to a strict pivotal category.
\end{prop}

Let $F: \oVec \to \cD$ be a pivotal equivalence to $\cD$, a strict pivotal category. Write $F:\ g:=\C_g \mapsto g:=F(\C_g)$ on simples, and denote the dual of $g$ in $\cD$ by $g^*$, then $g^*$ is isomorphic to $g^{-1}$. Let $\cD_0$ denote the full subcategory of $\cD$ whose objects are words in $\{g,g^*\}$. Then $\cD = \vertchar{c}(\cD_0)$, the Cauchy completion of $\cD_0.$

\begin{lem}
    The following objects and morphisms generate $\cD_0$, since every morphism space in $\cD_0$ is 1-dimensional. 
\begin{enumerate}[label=\textup{(O\arabic*)}]
\item
\label{Obj:1}
For each $g \in G,$ $g \neq \mathbf{1},$ we have objects $g:\
    \tikzmath{
    \draw[thick, blue, ->](0,0) -- (0,.6);
    \draw[thick, blue](0,.6) --(0,.8);
    \node[blue] at (.3, 0.3) {$g$};
    }$
    and
    $g^*:\
    \tikzmath{
    \draw[thick, blue, ->](0,.8) -- (0,.5);
    \draw[thick, blue](0,.5) --(0,0);
    \node[blue] at (.3, 0.3) {$g^*$};
    }$. We use a dotted strand,  
    $\tikzmath{
    \draw[thick, blue, dotted](0,0) -- (0,.8);}$\ , 
    or empty strand to denote $\mathbf{1}.$
\end{enumerate}   

\begin{enumerate}[label=\textup{(M\arabic*)}]


\item 
\label{Mor: trivalent}
For each pair $g,h \in G$, choose an isomorphism
$\mu_{g,h}: gh \rightarrow g \otimes h$
$
\tikzmath{
\draw[thick, blue,-->-] (0,0) -- (-.4,.4) node[above]{$\scriptstyle g$};
\draw[thick,blue, -->-] (0,0) -- (.4,.4) node[above]{$\scriptstyle h$};
\draw[thick, blue, ->-] (0,-.4) node[below]{$\scriptstyle gh$} -- (0,0);
\filldraw[blue] (0,0) node[left]{$\scriptstyle \mu_{g,h}$} circle (.0cm);
}$ to be the inverse of the tensorator of $F$, and the tensorator denoted by
$\tikzmath{
    \draw[thick, blue, -->-] (0,0) -- node[above,yshift=.1cm] {$\scriptstyle gh$}(0,0.4);
    \draw[thick, blue, ->] (-0.4,-0.4) node[below] {$\scriptstyle g$} -- (-0.2,-0.2);
    \draw[thick, blue] (-0.2, -0.2) -- (0,0);
    \draw[thick, blue, ->] (0.4,-0.4) node[below, xshift=.1cm] {$\scriptstyle h$} -- (0.2,-0.2);
    \draw[thick, blue] (0.2, -0.2) -- (0,0);
     \node[blue] at (.5,0.1) {$\scriptstyle \mu_{g,h}^{-1}$};
    }$.
    When $h=g^{-1}$, we denote the tensorator by
    $\tikzmath{
\draw[thick, blue, ->-] (-.4,-.4) -- (0,0);
\draw[thick,blue,->-] (.4,-.4) -- (0,0); 
\node[blue] at (-0.6, -.3) {$\scriptstyle g$};
\node[blue] at (.7,-.3) {$\scriptstyle g^{-1}$};
}$.

\item 
\label{Mor: cup and cap}
For each pair of duals $g, g^*$, $g \neq \mathbf{1},$ we have isomorphisms $\coev_{g^*} $
$\tikzmath{
\draw[thick, blue, ->] (0,0) arc(-180:0:0.3);
\node[blue] at (-0.2, 0) {$\scriptstyle g^*$};
\node[blue] at (0.8, 0) {$\scriptstyle g$};
}$,
$\ev_{g^*}$
$\tikzmath{
\draw[thick, blue, ->] (0,0) arc(180:0:0.3);
\node[blue] at (-0.2, 0) {$\scriptstyle g$};
\node[blue] at (0.9, 0) {$\scriptstyle g^*$};
}$,
$\coev_{g}$
$\tikzmath{
\draw[thick, blue, <-] (0,0) arc(-180:0:0.3);
\node[blue] at (-0.2, 0) {$\scriptstyle g$};
\node[blue] at (0.8, 0) {$\scriptstyle g^*$};
}$,
$\ev_{g}$
$\tikzmath{
\draw[thick, blue, <-] (0,0) arc(180:0:0.3);
\node[blue] at (-0.25, 0) {$\scriptstyle g^*$};
\node[blue] at (0.9, 0) {$\scriptstyle g$};
}$
that witness duality.
\end{enumerate}   
\end{lem}

\vspace{.1cm}

\begin{lem}\label{lem:R1-R5}
The following relations hold in $\cD_0$, for all $(g,h,k) \in G^3$ and $g \in G$:

\begin{enumerate}[label=\textup{(R\arabic*)}]
\item
\label{R:associativity}
$\tikzmath{
\draw[thick, blue] (.3,.3) node[above]{$\scriptstyle k$} -- (0,0) -- (-.3,.3) node[above]{$\scriptstyle h$};
\draw[thick, blue] (0,0) -- (-.3,-.3) -- (-.9,.3) node[above]{$\scriptstyle g$};
\draw[thick, blue] (-.3,-.3) -- (-.3,-.6) node[below]{$\scriptstyle ghk$};
\filldraw[blue] (0,0) node[right]{$\scriptstyle \mu_{h,k}$} circle (.0cm);
\filldraw[blue] (-.3,-.3) node[right]{$\scriptstyle \mu_{g,hk}$} circle (.0cm);
}
=
\omega(g,h,k)
\tikzmath{
\draw[thick, blue] (-.3,.3) node[above]{$\scriptstyle g$} -- (0,0) -- (.3,.3) node[above]{$\scriptstyle h$};
\draw[thick, blue] (0,0) -- (.3,-.3) -- (.9,.3) node[above]{$\scriptstyle k$};
\draw[thick, blue] (.3,-.3) -- (.3,-.6) node[below]{$\scriptstyle ghk$};
\filldraw[blue] (0,0) node[left]{$\scriptstyle \mu_{g,h}$} circle (.0cm);
\filldraw[blue] (.3,-.3) node[left]{$\scriptstyle \mu_{gh,k}$} circle (.0cm);
}
$, for some $\ \omega \in [\omega]$.

\item 
\label{R: id}
 $\tikzmath{
\draw[thick, blue, ->] (0,0) -- (0,0.4);
\draw[thick, blue, ->] (-0.4, -0.4) -- (-0.2, -0.2);
\draw[thick, blue] (-.2,-.2) -- (0,0);
\draw[thick, blue, dotted] (0.4,-0.4) -- (0,0);
\node[blue] at (-0.2, 0.2) {$\scriptstyle g$};
}$
=
$\tikzmath{
\draw[thick, blue, ->](0,-0.4) -- (0,0.4);
\node[blue] at (0.2,0) {$\scriptstyle g$};
}$ and  
$\tikzmath{
\draw[thick, blue, ->] (0,0) -- (0,0.4);
\draw[thick, blue, ->] (0.4, -0.4) -- (0.2, -0.2);
\draw[thick, blue] (.2,-.2) -- (0,0);
\draw[thick, blue, dotted] (-0.4,-0.4) -- (0,0);
\node[blue] at (-0.2, 0.2) {$\scriptstyle g$};
}$
=
$\tikzmath{
\draw[thick, blue, ->](0,-0.4) -- (0,0.4);
\node[blue] at (0.2,0) {$\scriptstyle g$};
}.$

\vspace{.3cm}

\item 
\label{R:zigzag}
$\tikzmath{
\draw[thick, blue] (-.6,-.5) -- node[left]{$\scriptstyle g^*$} (-.6,0) arc(180:0:.3) ;
\draw[thick, blue, <-] (0,0) node[right] {$\scriptstyle g$} arc(-180:0:.3) -- node[right]{$\scriptstyle g^*$} (.6,.5);
}
=
\tikzmath{
\draw[thick, blue, ->] (0,.5) -- node[right]{$\scriptstyle g^*$} (0,0);
\draw[thick, blue] (0,0) -- (0,-.5);
},$
$\tikzmath{
\draw[thick, blue] (-.6,.5) -- node[left]{$\scriptstyle g$} (-.6,0) arc(-180:0:.3) ;
\draw[thick, blue, <-] (0,0) node[right] {$\scriptstyle g^*$} arc(180:0:.3) -- node[right]{$\scriptstyle g$} (.6,-.5);
}
=
\tikzmath{
\draw[thick, blue] (0,.5) -- node[right]{$\scriptstyle g$} (0,0);
\draw[thick, blue, <-] (0,0) -- (0,-.5);
},$
$\tikzmath{
\draw[thick, blue,->] (-.6,-.5) -- node[left]{$\scriptstyle g$} (-.6,0) arc(180:0:.3) ;
\draw[thick, blue] (0,0) node[right] {$\scriptstyle g^*$} arc(-180:0:.3) -- node[right]{$\scriptstyle g$} (.6,.5);
}=
\tikzmath{
\draw[thick, blue] (0,.5) -- node[right]{$\scriptstyle g$} (0,0);
\draw[thick, blue, <-] (0,0) -- (0,-.5);
},$
$\tikzmath{
\draw[thick, blue,->] (-.6,.5) -- node[left]{$\scriptstyle g^*$} (-.6,0) arc(-180:0:.3) ;
\draw[thick, blue] (0,0) node[right] {$\scriptstyle g$} arc(180:0:.3) -- node[right]{$\scriptstyle g^*$} (.6,-.5);
}
=
\tikzmath{
\draw[thick, blue,->] (0,.5) -- node[right]{$\scriptstyle g^*$} (0,0);
\draw[thick, blue] (0,0) -- (0,-.5);
}.$

\vspace{.3cm}

\item 
\label{R:bubble}
$\tikzmath{
\draw[thick, blue, <-] (-.3,0) arc (180:-180:.3cm);
\node[blue] at (-.5,0) {$\scriptstyle g^*$};
\node[blue] at (.5,0) {$\scriptstyle g$};
}
=
\pi(g^{-1})$
and
$\tikzmath{
\draw[thick, blue, ->] (-.3,0) arc (180:-180:.3cm);
\node[blue] at (-.5,0) {$\scriptstyle g$};
\node[blue] at (.5,0) {$\scriptstyle g^*$};
}
=
\pi(g)
$.

\vspace{.3cm}

\item 
\label{R: pivotal}
$\tikzmath{
\draw[thick, blue] (-.6,-.6) -- (-.4,-.4);
\draw[thick, blue, ->-] (-.4,-.4) -- (0,0);
\draw[thick,blue,->-] (.4,-.4) -- (0,0); \draw[thick,blue] (.4,-.4) arc(-135:0:.4);
\draw[thick, blue, ->] ({(4+sqrt(2))/5},.2) -- ({(4+sqrt(2))/5},{0.4*(1/sqrt(2)-1)});
\node[blue] at (-0.6, -.3) {$\scriptstyle g^{-1}$};
\node[blue] at (.1,-.4) {$\scriptstyle g$};
\node[blue] at (1.4,.2) {$\scriptstyle g^*$};
}
=
\omega(g^{-1}, g, g^{-1}) \pi(g) 
\tikzmath{
\draw[thick, blue] (.6,-.6) -- (.4,-.4);
\draw[thick, blue, ->-] (-.4,-.4) -- (0,0);
\draw[thick,blue,->-] (.4,-.4) -- (0,0); \draw[thick,blue] (-.4,-.4) arc(-45:-180:.4);
\draw[thick, blue, ->] (-{(4+sqrt(2))/5},.2) -- (-{(4+sqrt(2))/5},{0.4*(1/sqrt(2)-1)});
\node[blue] at (0.75, -.3) {$\scriptstyle g^{-1}$};
\node[blue] at (-.1,-.4) {$\scriptstyle g$};
\node[blue] at (-1.3,.2) {$\scriptstyle g^*$};
}
$
\end{enumerate}  

\comment{
Moreover, we may assume the tensorators are trivial when  $(g,h)= (\mathbf{1},g)$ or $(g,\mathbf{1})$}

\comment{
\item 
\label{R: replacement}
$\tikzmath{
\draw[thick, blue, <-] (0,0) arc(180:0:0.3);
\node[blue] at (-0.2, 0) {$\scriptstyle g^*$};
\node[blue] at (0.9, 0) {$\scriptstyle g$};
}$
= 
$\tikzmath{
\draw[thick, blue,dotted] (0,0) -- (0,0.3);
\draw[thick, blue,-<-] (0,0) --  (-0.4, -0.4);
\draw[thick, blue,->-](0.4,-0.9) to (0,0);
\draw[thick, blue,->] (-0.4,-0.4) -- (-0.4, -0.9);
\filldraw[blue] (-.4,-.4) node[left]{$\scriptstyle \delta_g$} circle (.05cm);
\node[blue] at (-0.6, -0.7) {$\scriptstyle g^*$};
\node[blue] at (-0.35,-0.05) {$\scriptstyle g^{-1}$};
\node[blue] at (.55,-.8) {$\scriptstyle g $};
}$.}

\end{lem}


\begin{proof}
\ref{R:associativity} Recall that $F: \oVec \to \cD$ is a monoidal functor, and that we chose $\mu_{g,h}: gh \to g \otimes h $ to be the inverse of the tensorator of $F.$ Then \ref{R:associativity} follows from monoidality of $F$ and strictness of $\cD$. \\

\ref{R: id} We may assume the tensorators are trivial when  $(g,h)= (\mathbf{1},g)$ or $(g,\mathbf{1})$.\\

\ref{R:zigzag} This is immediate from our choice of cups and caps. \\

\ref{R:bubble} The left/right quantum dimensions of $g$ in skeletal $\oVec$ are given by \eqref{eq: quantum dims of simple}. By \cite[lemma 2.14]{penneys2018unitarydualfunctorsunitary}, $F$ preserves quantum dimensions. Hence \ref{R:bubble} hold. The right quantum dimensions
$\tikzmath{
\draw[thick, blue, ->] (-.3,0) arc (180:-180:.3cm);
\node[blue] at (-.5,0) {$\scriptstyle g$};
\node[blue] at (.5,0) {$\scriptstyle g^*$};
}$ $= \pi(g)$ of all $g \in G$ imply their left quantum dimensions, since $ \dim_L(g) = \dim_R(g^*) = \dim_R(g^{-1})$.\\

\ref{R: pivotal} This unpacks the pivotality condition of $F$ in Definition \ref{pivfunctor}, assuming our duality and pivotal isomorphism choices as in \eqref{eq:duality-choice}. 
\end{proof}
\begin{rem}
Notice that in our case, pivotality  condition \ref{R: pivotal} is generated by \ref{R:associativity} -- \ref{R:bubble}:

Use \ref{R:associativity} on $(g,g^{-1},g)$ and \ref{R: id} to deduce
\[
\tikzmath{
\draw[thick, blue, dotted] (0,0) -- (0,.4);
\draw[thick, blue,->](0,.4) -- node[left,xshift=-.1cm]{$\scriptstyle g$} (-.4,.8);
\draw[thick, blue] (-.4,.8) -- (-.6,1);
\draw[thick, blue,->] (0,.4)-- (.4,.8) node[left,xshift = .2cm, yshift=.2cm]{$\scriptstyle g^{-1}$};
\draw[thick, blue ] (.4,.8) -- (.8,1.2);
\draw[thick, blue] (.8,1.2)--(1,1);
\draw[thick, blue,<-] (1,1) -- node[right, xshift=-.1cm, yshift =.1cm]{$\scriptstyle g$} (1.4,0.6);
\draw[thick, blue, dotted] (.8,1.2)--(.8, 1.6);
} = \omega(g^{-1},g,g^{-1})\
\tikzmath{
\draw[thick, blue, ->-] (0,0) -- (0,1) node[right]{$\scriptstyle g$};
}.
\]
Compose with $\coev_g$ and $\ev_{g^*}$ to get closed loops
\[
\tikzmath{
\draw[thick, blue, ->-] (0,0) -- (0,.8) node[left,xshift = 0cm, yshift=.1cm]{$\scriptstyle g$};
\draw[thick, blue,->-] (0,0)-- (.8,.8) node[left,xshift = -.0cm, yshift=0cm]{$\scriptstyle g^{-1}$};
\draw[thick, blue, ->-] (0.8,0.2) -- (0.8,.8) node[right,xshift = 0cm, yshift=-.2cm]{$\scriptstyle g$};
\draw[thick,blue] (.8,0.2) arc(-180:0:.3cm);
\draw[thick, blue, -<-] (1.4,0.2) -- (1.4,.8) node[right,xshift = 0cm, yshift=-.2cm]{$\scriptstyle g^*$};
\draw[thick,blue] (1.4,.8) arc(0:180:.7cm);
}
=
\omega(g^{-1},g,g^{-1}) 
\tikzmath{
\draw[thick, blue, ->] (-.3,0) arc (180:-180:.3cm);
\node[blue] at (-.5,0) {$\scriptstyle g$};
\node[blue] at (.5,0) {$\scriptstyle g^*$};
}
= \omega(g^{-1},g,g^{-1}) \pi(g).
\]
Precompose with $\mu_{g,g^{-1}}^{-1}$ on both sides to get
\[
\tikzmath{
\draw[thick, blue, ->-] (-.4,0) -- (-.4,.8) node[left,xshift = 0cm, yshift=.1cm]{$\scriptstyle g$};
\draw[thick, blue,->-] (0,0)-- (.8,.8) node[left,xshift = -.0cm, yshift=0cm]{$\scriptstyle g^{-1}$};
\draw[thick, blue, ->-] (0.8,0.2) -- (0.8,.8) node[right,xshift = 0cm, yshift=-.2cm]{$\scriptstyle g$};
\draw[thick,blue] (.8,0.2) arc(-180:0:.3cm);
\draw[thick, blue, -<-] (1.4,0.2) -- (1.4,.8) node[right,xshift = 0cm, yshift=-.2cm]{$\scriptstyle g^*$};
\draw[thick,blue] (1.4,.8) arc(0:180:.9cm);
}
= \omega(g^{-1},g,g^{-1}) \pi(g)\
\tikzmath{
\draw[thick,blue,->-] (0,0) -- (.6,.6) node[left,xshift = -.2cm, yshift=-.8cm]{$\scriptstyle g$};
\draw[thick,blue,->-] (1.2,0) -- (.6,.6) node[right,xshift = .2cm, yshift=-.8cm]{$\scriptstyle g^{-1}$};
}.
\]
Finally, compose with $\coev_{g^*}$ on both sides, and by \ref{R:zigzag}, this is equivalent to
\[
\tikzmath{
\draw[thick, blue] (-.6,-.6) -- (-.4,-.4);
\draw[thick, blue, ->-] (-.4,-.4) -- (0,0);
\draw[thick,blue,->-] (.4,-.4) -- (0,0); \draw[thick,blue] (.4,-.4) arc(-135:0:.4);
\draw[thick, blue, ->] ({(4+sqrt(2))/5},.2) -- ({(4+sqrt(2))/5},{0.4*(1/sqrt(2)-1)});
\node[blue] at (-0.6, -.3) {$\scriptstyle g^{-1}$};
\node[blue] at (.1,-.4) {$\scriptstyle g$};
\node[blue] at (1.4,.2) {$\scriptstyle g^*$};
}
=
\omega(g^{-1}, g, g^{-1}) \pi(g) 
\tikzmath{
\draw[thick, blue] (.6,-.6) -- (.4,-.4);
\draw[thick, blue, ->-] (-.4,-.4) -- (0,0);
\draw[thick,blue,->-] (.4,-.4) -- (0,0); \draw[thick,blue] (-.4,-.4) arc(-45:-180:.4);
\draw[thick, blue, ->] (-{(4+sqrt(2))/5},.2) -- (-{(4+sqrt(2))/5},{0.4*(1/sqrt(2)-1)});
\node[blue] at (0.75, -.3) {$\scriptstyle g^{-1}$};
\node[blue] at (-.1,-.4) {$\scriptstyle g$};
\node[blue] at (-1.3,.2) {$\scriptstyle g^*$};
}.
\]
In fact, given a monoidal functor $F$ from a pointed pivotal category $\cX$ to any pivotal category $\cY,$ if $F$ preserves quantum trace, then $F$ is pivotal. Notice that the proof for ``pivotal functors preserve quantum traces" \cite[Lemma 2.14]{penneys2018unitarydualfunctorsunitary} is an if and only if statement when the source is pointed.

\end{rem}


\begin{lem}\label{lem: replacements}
   Define isomorphisms
   \begin{enumerate}[label=\textup{(M\arabic*)}]\addtocounter{enumi}{+2}
   \item
   \label{Mor: dot}
   $\delta_g = 
   \tikzmath{
    \draw[thick, blue, ->] (0, -0.4) node[below] {$\scriptstyle g^{-1}$} -- (0,-0.1);
    \draw[thick, blue] (0,-.1) -- (0,0);
    \draw[thick, blue, ->] (0, .4) node[above] {$\scriptstyle g^{*}$} -- (0,.1);
    \draw[thick, blue] (0,.1) -- (0,0);
    \filldraw[blue] (0,0) node[left]{$\scriptstyle \delta_{g}$} circle (.05cm);
    }
    \coloneq
    \tikzmath{
    \draw[thick, blue] (-.6,-.6) -- (-.4,-.4);
    \draw[thick, blue, ->-] (-.4,-.4) -- (0,0);
    \draw[thick,blue,->-] (.4,-.4) -- (0,0); \draw[thick,blue] (.4,-.4) arc(-135:0:.4);
    \draw[thick, blue, ->] ({(4+sqrt(2))/5},.2) -- ({(4+sqrt(2))/5},{0.4*(1/sqrt(2)-1)});
    \node[blue] at (-0.6, -.3) {$\scriptstyle g^{-1}$};
    \node[blue] at (.1,-.4) {$\scriptstyle g$};
    \node[blue] at (1.4,.2) {$\scriptstyle g^*$};
    }$ for every $g \in G$ to be the canonical isomorphism in Definition \ref{pivfunctor}. Let $
    \tikzmath{
    \draw[thick, blue, ->] (0,0) -- (0,-.4) node[below]{$\scriptstyle g^*$};
    \draw[thick,blue, ->] (0,0) -- (0,.4) node[above]{$\scriptstyle g^{-1}$};
    \filldraw[blue] (0,0) node[left]{$\scriptstyle (\delta_{g})^{-1}$} circle (.05cm);
    }$ denote its inverse.
    \end{enumerate}
    Then the following 4 replacement relations can be deduced:

\begin{enumerate}[label=\textup{(R\arabic*)}]\addtocounter{enumi}{+5}

\item
\label{R: replacement}
$\tikzmath{
\draw[thick, blue, <-] (0,0) arc(180:0:0.3);
\node[blue] at (-0.2, 0) {$\scriptstyle g^*$};
\node[blue] at (0.9, 0) {$\scriptstyle g$};
}$
= 
$\tikzmath{
\draw[thick, blue,dotted] (0,0) -- (0,0.3);
\draw[thick, blue,-<-] (0,0) --  (-0.4, -0.4);
\draw[thick, blue,->-](0.4,-0.9) to (0,0);
\draw[thick, blue,->] (-0.4,-0.4) -- (-0.4, -0.9);
\filldraw[blue] (-.4,-.4)  circle (.05cm);
\node[blue] at (-0.6, -0.7) {$\scriptstyle g^*$};
\node[blue] at (-0.35,-0.05) {$\scriptstyle g^{-1}$};
\node[blue] at (.55,-.8) {$\scriptstyle g $};
}$.

\item \label{R:replacement2}
$\tikzmath{
    \draw[thick, blue, ->] (0,0) node[right] {$\scriptstyle g^*$} arc(0:-180:.3) node[left] {$\scriptstyle g$};
    }
    = \omega(g, g^{-1},g) 
    \tikzmath{
    \draw[thick, blue, ->] (0,0) -- node[right] {$\scriptstyle g^{-1}$} (.2,.2);
    \draw[thick, blue] (.2,.2) -- (.3,.3);
    \draw[thick, blue, -->-] (0,0) --  (-.5,.8) node[above] {$\scriptstyle g$};
    \draw[thick, blue, dotted] (0,0) -- (0, -.3);
    \draw[thick, blue] (.3,.3) -- (.3, .5);
    \draw[thick, blue, ->] (.3, .8) node[above] {$\scriptstyle g^*$} -- (.3,.5);
    \filldraw[blue] (0.3,0.3) circle (.05cm);
    }$

\vspace{.2cm}

    \item \label{R:replacement3}
    $\tikzmath{
\draw[thick, blue, ->] (0,0) arc(180:0:0.3);
\node[blue] at (-0.2, 0) {$\scriptstyle g$};
\node[blue] at (0.9, 0) {$\scriptstyle g^*$};
}$
= $\frac{\pi(g)}{\omega(g,g^{-1},g)}$
$\tikzmath{
\draw[thick, blue,dotted] (0,0) -- (0,0.3);
\draw[thick, blue] (0,0) -- (.1,-.1);
\draw[thick, blue, ->-](-.5,-.8) -- node[left]{$\scriptstyle g$} (0,0);
\draw[thick, blue, ->] (.3, -.3) -- node[right, yshift=.1cm]{$\scriptstyle g^{-1}$} (.1,-.1);
\draw[thick, blue, ->](.3,-.3) -- node[right]{$\scriptstyle g^*$} (.3, -.6);
\filldraw[blue] (.3,-.3) circle (.05cm);
\draw[thick, blue](.3,-.6) -- (.3,-.8);
}$

\item \label{R:replacement4}
$\tikzmath{
    \draw[thick, blue, ->] (0,0) arc(-180:0:.3);
    \node[blue] at (-.2,0) {$\scriptstyle g^*$};
    \node[blue] at (.9,0) {$\scriptstyle g$};
    }
    =
    \pi(g^{-1})
    \tikzmath{
    \draw[thick, blue, -->-] (0,0) -- node[left, yshift=-.1cm] {$\scriptstyle g^{-1}$} (-.3,.3);
    \draw[thick, blue, -->-] (0,0) -- node[right] {$\scriptstyle g$} (.5,.8);
    \draw[thick, blue, dotted] (0,0) -- (0, -.3);
    \draw[thick, blue, -<-] (-.3, .3) -- node[left, yshift=.2cm] {$\scriptstyle g^*$} (-.3, .8);
    \filldraw[blue] (-.3,.3) circle (.05cm);
    }$

\end{enumerate}
\end{lem}
\begin{proof}
They all follow promptly from the definition of $\delta_g$ and relations \ref{R:associativity} -- \ref{R:bubble}.
\end{proof}

\section{A presentation of the strictified $\oVec$ }
\begin{defn} (c.f. \cite[2.1]{etingof2019semisimplificationtensorcategories})
   For $\cC$ a linear monoidal category, a \textbf{tensor ideal} $I$ in $\cC$ is a collection of subspaces $I(x,y) \subset \cC(x,y)$ for all $x,y \in \cC$ such that for all $x,y,z,t \in \cC$:
   \begin{enumerate}
       \item \label{i1} for $f \in I(x,y)$ and $g \in \cC(y,z)$, $h \in \cC(z,x)$, we have $f \circ h \in I (z,y)$ and $g \circ f \in I(x,z)$;
       \item \label{i2} for $f \in I(x,y),$ $g \in \cC(z,t)$, we have $f \otimes g \in I(x \otimes z, y \otimes t)$ and $g \otimes f \in I(z \otimes x , t \otimes y)$.
   \end{enumerate}
   When $I$ is a tensor ideal, the quotient $\cC/I$ defines a linear monoidal category: objects of $\cC/I$ are the objects of $\cC;$ morphism space $\cC/I(x,y) \coloneq \cC(x,y)/I(x,y);$ composition and tensor product of morphisms are inherited from $\cC.$
\end{defn}

\begin{defn} (c.f. \cite[2.2]{Daniel-Cain})
    A \textbf{presentation} of a pivotal category $\cC$ is a set of generating objects $X$, generating morphisms $F$ between them, and relations $R$ satisfied in $\cC$ such that 
    \[ \cC \simeq \vertchar{c}(\cC(X,F)/\cR),\]
     where $\cC(X,F)$ is the free strict pivotal category generated by objects $X$ and morphisms $F$, $\cR$ is the smallest tensor ideal of $\cC(X,F)$ containing $R$, and $\vertchar{c}(\cC)$ denotes the Cauchy completion of a category $\cC.$ 
\end{defn}

\begin{thm}\label{mainthm}
Take $X$ to be the set of formal objects \ref{Obj:1}, $F$ to be the set of isomorphisms \ref{Mor: trivalent} -- \ref{Mor: cup and cap}, and $R$ to be the relations \ref{R:associativity} -- \ref{R:bubble}. Then $\cC(X,F)/\cR \cong \cD_0$, i.e., $(X,F,R)$ is a presentation of $\oVec.$
\end{thm}

Since \ref{Mor: dot} and \ref{R: pivotal} -- \ref{R:replacement4} are generated by \ref{Mor: trivalent} -- \ref{Mor: cup and cap} and \ref{R:associativity} -- \ref{R:bubble}, we may use them in proving Theorem \ref{mainthm}.
\begin{defn}
    For a diagram with only trivalent vertices \ref{Mor: trivalent}, i.e., without rotation of the strands, a \textbf{rigid isotopy} is a planar isotopy which does not introduce or cancel any local maximum or minimum.
\end{defn}

\begin{lem}\label{lem:ngon}
    Relations \ref{R:associativity}--\ref{R:replacement4} are sufficient to reduce the number of vertices by 2 for each bi-gon, triangle, square, and pentagon only consisting of trivalent vertices \ref{Mor: trivalent}. 
\end{lem}

\begin{proof}We omit the arrows since all are pointing upward.\\
    \begin{itemize}
        \item 
        bigon: We have
     \[\tikzmath{
    \draw[thick, blue] (-.3,0) arc(180:-180:.3) ;
    \draw[thick,blue] (0,.3) -- (0,.6);
    \draw[thick,blue](0,-.3) -- (0,-.6);
    \node[blue] at (0,-.8) {$\scriptstyle gh$};
    \node[blue] at (0,.8) {$\scriptstyle gh$};
    \node[blue] at (-.45,0) {$\scriptstyle g$};
    \node[blue] at (.45,0) {$\scriptstyle h$};
    }
    =
    \tikzmath{
    \draw[thick, blue] (0,-1) -- node[right] {$\scriptstyle gh$}    (0,.2);
    }\]  
    by definitions of trivalent generating morphisms and their inverses in \ref{Mor: trivalent};

        \item 
        triangle:
        We have

        \[\scalebox{.8}{$\tikzmath{
        \draw[thick, blue] (0,-.5)node[below]{$\scriptstyle ghk$}-- (0,0)   node[left]{$\scriptstyle gh$} --(-.4,.4) -- (-.9,.9) node[above]{$\scriptstyle g$};
        \draw[thick, blue] (0,0)-- node[right, xshift=-.2cm, yshift=-.2cm]{$\scriptstyle k$} (.7,.7) -- (1,1) node[above]{$\scriptstyle hk$};
        \draw[thick, blue] (-.4,.4)--(.7,.7);
        \node[blue] at (0,.7) {$\scriptstyle h$};
        }
        =
        \omega(g,h,k)^{-1}
        \tikzmath{
        \draw[thick, blue] (0,-.5) node[below]{$\scriptstyle ghk$} -- (0,0) --(-.9,0.9) node[above]{$\scriptstyle g$} ;
        \draw[thick, blue] (0,0) -- (.3,.3);
        \draw[thick,blue] (.8,.8) -- (1,1)node[above]{$\scriptstyle hk$};
        \draw[thick,blue] (.8,.8) to[bend right=40] (.3,.3);
        \draw[thick,blue] (.8,.8) to[bend left=40] (.3,.3);
        \node[blue] at (.3,0) {$\scriptstyle hk$};
        \node[blue] at (.45,.9) {$\scriptstyle h$};
        \node[blue] at (.75,.25) {$\scriptstyle k$};
        }
        =
        \omega(g,h,k)^{-1}
        \tikzmath{
        \draw[thick, blue] (0,-.5) node[below]{$\scriptstyle ghk$}-- (0,0) --(-.9,0.9) node[above]{$\scriptstyle g$} ;
        \draw[thick, blue] (0,0) --   (1,1)node[above]{$\scriptstyle hk$};
        }$}\]
        by relation \ref{R:associativity} and the bi-gon reduction above. The number of vertices is reduced by 2. Every other triangle is a vertical or horizontal symmetry of the triangle above up to rigid isotopy, and can be reduced similarly.
        
        \item 
        square:
        We notice that by pre-composing with isomorphism $\mu_{gh,k}$, the folllowing relation is a triangle reduction proved above:

\begin{enumerate}[label=\textup{(R\arabic*)}]\addtocounter{enumi}{+9}

\item \label{R:H shape}

$\tikzmath{
\draw[thick, blue] (0,0) node[below]{$\scriptstyle gh$} -- (0,1) node[above]{$\scriptstyle g$};
\draw[thick, blue] (.8,0) node[below]{$\scriptstyle k$} -- (0.8,1) node[above]{$\scriptstyle hk$};
\draw[thick, blue](0,.3) -- (.8,.7);
\node[blue] at (.4,.7) {$\scriptstyle h$};
}
=
\omega(g,h,k)^{-1}
\tikzmath{
\draw[thick,blue] (0,0) node[below]{$\scriptstyle gh$} -- (.4,.4) -- node[right]{$\scriptstyle ghk$} (.4,.8) -- (0,1.2) node[above]{$\scriptstyle g$};
\draw[thick, blue] (.8,0) node[below]{$\scriptstyle k$} --(.4,.4);
\draw[thick, blue] (.4,.8) -- (.8,1.2) node[above]{$\scriptstyle hk$};
}$
\end{enumerate}

        Every square or its vertical/horizontal reflection is rigid isotopic to one of the following:

        \[\scalebox{1.2}{$\tikzmath{
        \draw[thick,blue](0,-.2)--(0,1.2);
        \draw[thick, blue](.8,-.2) -- (.8,1.2);
        \draw[thick,blue](0,.1)--(.8, .5);
        \draw[thick,blue](0,.5)--(.8,.9);
        \draw[thick,blue,dotted] (-.2,-.1)-- (.4,-.1) -- (.4,.9) -- (-.2,.9) -- (-.2,-.1);
        }
        \qquad
        \tikzmath{
        \draw[thick,blue](0,-.2)--(0,1.2);
        \draw[thick, blue](.8,-.2) -- (.8,1.2);
        \draw[thick,blue](0,.3)--(.8, 0);
        \draw[thick,blue](0,.5)--(.8,.9);
        \draw[thick,blue,dotted] (-.2,.4)-- (1,.4) -- (1,1) -- (-.2,1) -- (-.2,.4);
        }
        \qquad
        \tikzmath{
        \draw[thick,blue](-.9,-.9) -- (0,0) -- (0,.5);
        \draw[thick,blue](.9,-.9) -- (0,0);
        \draw[thick,blue](-.2,-.2) -- (0,-.8) -- (0,-1);
        \draw[thick,blue](0,-.8) -- (.5,-.5);
        \draw[thick,blue,dotted](-.5,-.4)--(.4,-.4) -- (.4,.2)--(-.5,.2)--(-.5,-.4);
        }
        \qquad
        \tikzmath{
        \draw[thick,blue](-.9,-.9) -- (0,0) -- (0,.5);
        \draw[thick,blue](.9,-.9) -- (0,0);
        \draw[thick,blue](-.2,-.2) -- (0,-.6) -- (0,-1);
        \draw[thick,blue](0,-.6) -- (.8,-.8);
        \draw[thick,blue,dotted](-.5,-.4)--(.4,-.4) -- (.4,.2)--(-.5,.2)--(-.5,-.4);}$}
        \]

        which can be reduced to a triangle while the number of vertices remains the same, by applying \ref{R:associativity} or \ref{R:H shape} to the boxed region.
        \item 
        pentagon: 
        Each pentagon or its vertical/horizontal reflection is rigid isotopic one of the following:
        
        \[\tikzmath{
        \draw[thick,blue](0,-.2) -- (0,1.5);
        \draw[thick,blue](1,-.2) -- (1,1.5);
        \draw[thick,blue](0,.9) -- (1,1.2);
        \draw[thick,blue](0,.7)--(.5,.3) -- (1,.7);
        \draw[thick,blue](.5,.3) -- (.5,-.2);
        \draw[thick,blue,dotted](-.2,.8) -- (1.2,.8) -- (1.2,1.4) -- (-.2,1.4) -- (-.2,.8);
        }
        \qquad
        \tikzmath{
        \draw[thick,blue](0,-.2) -- (0,1.5);
        \draw[thick,blue](1,-.2) -- (1,1.5);
        \draw[thick,blue](0,.9) -- (1,1.2);
        \draw[thick,blue](0,.7)--(.5,.3) -- (1,.1);
        \draw[thick,blue](.5,.3) -- (.5,-.2);
        \draw[thick,blue,dotted](-.2,.8) -- (1.2,.8) -- (1.2,1.4) -- (-.2,1.4) -- (-.2,.8);
        }
        \qquad
        \tikzmath{
        \draw[thick,blue](0,-.2) -- (0,1.5);
        \draw[thick,blue](1,-.2) -- (1,1.5);
        \draw[thick,blue](0,.9) -- (1,1.2);
        \draw[thick,blue](0,.1)--(.5,.3) -- (1,.7);
        \draw[thick,blue](.5,.3) -- (.5,-.2);
        \draw[thick,blue,dotted](-.2,.8) -- (1.2,.8) -- (1.2,1.4) -- (-.2,1.4) -- (-.2,.8);
        }
        \qquad
        \tikzmath{
        \draw[thick,blue](-1.2,-1.2) -- (0,0) -- (1.2,-1.2);
        \draw[thick,blue](0,0) -- (0,.3);
        \draw[thick,blue](-.6,-.6) -- (-.2,-.9) -- (.2,-.7) -- (.3,-.3);
        \draw[thick,blue](-.2,-.9) -- (-.2,-1.2);
        \draw[thick,blue](.2,-.7) -- (.2,-1.2);
        \draw[thick,blue,dotted](-.7,-.5) -- (.7,-.5)--(.7,.2) -- (-.7,.2) -- (-.7,-.5);
        }
        \qquad
        \tikzmath{
        \draw[thick,blue](-1.2,-1.2) -- (0,0) -- (1.2,-1.2);
        \draw[thick,blue](0,0) -- (0,.3);
        \draw[thick,blue](-1,-1) -- (-.2,-.9) -- (.2,-.7) -- (.3,-.3);
        \draw[thick,blue](-.2,-.9) -- (-.2,-1.2);
        \draw[thick,blue](.2,-.7) -- (.2,-1.2);
        \draw[thick,blue,dotted](-.7,-.5) -- (.7,-.5)--(.7,.2) -- (-.7,.2) -- (-.7,-.5);
        }
        \]
        which can be reduced to a square while the number of vertices remains the same, by applying \ref{R:associativity} or \ref{R:H shape} to the boxed region. \qedhere
    \end{itemize}
\end{proof}

\begin{lem}\label{lem: small face}
     Any trivalent diagram on a disk $D^2$ satisfying the following conditions has a pentagon or smaller face. \\
     (i). It has exactly two open edges intersecting the boundary of $D^2$.\\
     (ii). Every face in it is an $n$-gon for $n \geq 2$.
\end{lem}
\begin{proof}
 Let $s$ be such a diagram. We connect the source and target strands and embed this closed diagram on a sphere: by doing so we created two faces.

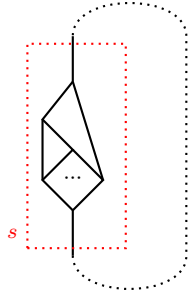
\[
\tikzmath{
\draw[thick](0,.2) -- (0,.8) -- (-.4,1.2)--(0,1.6)--(-.4,2)--(0,2.5) ;
\draw[thick](0,.8) -- (.4,1.2)--(0,1.6);
\draw[thick](-.4,1.2)--(-.4, 2);
\draw[thick](.4,1.2) -- (0,2.5) -- (0,3.1);
\draw[thick,  dotted] (0,.2) to[bend right=90] (1.5,.2) -- (1.5, 3.1) to[bend right=90] (0,3.1);
\draw[thick, red,dotted] (-.6,.3) -- (.7,.3) -- (.7,3) -- (-.6,3)--(-.6,.3);
\node[black] at (0,1.2) {$\scriptstyle \cdots$};
\node[red] at (-.8,0.5){$\scriptstyle s$};
}
\]

Denote the number of vertices by $|V|,$ the number of edges by $|E|,$ the number of faces by $|F|,$ the number of faces with $n$ edges by $|F_n|.$ By Euler's characteristic, $|V|-|E|+|F|=2,$ we have 
\[|F| =  \sum_n  |F_n|,\hspace{.3cm} |E| = \frac{1}{2}  \sum_n  n|F_n|,\hspace{.3cm}  |V| = \frac{1}{3}  \sum_n  n|F_n|,\] since each edge is shared by 2 faces and each vertex is shard by 3 faces. Thus the Euler's characteristic yields
\[
\sum_n \left( 1 - \frac{n}{6} \right) |F_n| = 2.
\]
We have $|F_1|=0$ by assumption. Each $n$-gon for $2 \leq n \leq 5$ contributes a positive weight of at most $\frac{2}{3}$ to this sum. Each hexagon contributes 0 weight. Each $n$-gon for $n \geq 7$ contributes a negative weight. Thus there are at least 3 pentagon or smaller faces in this closed diagram, i.e., there is at least one such face in $s.$
\end{proof}

\begin{prop}\label{lem:g_to_g}
    Relations \ref{R:associativity}--\ref{R:replacement4} are sufficient to evaluate any diagram in $ \Hom(g \to g)$ for any $g \in G$ in the strict pivotal category that $(X,F,R)$ represents.
\end{prop}
\begin{proof}
Let a diagram $s_0: g \to g$ be given. Use \ref{R: replacement} -- \ref{R:replacement4} to replace all the cup/caps \ref{Mor: cup and cap} in a diagram. Then for each $\delta_g$ that appears in the diagram, it needs to be in between two trivalent vertices, and there needs to exist a $\delta_g^{-1}$ adjacent to it for the diagram to make sense, since trivalent vertices are only defined on objects $\{g \}_{g \in G}$. Hence all the $\delta_g$ \ref{Mor: dot} are cancelled. For any 2-valent cusp in the diagram, add a dotted line to represent 1 and connect it to a neighboring edge, so that each vertex is trivalent. We can do so due to \ref{R: id}. We are left to reduce a diagram $s: g \to g$ with only trivalent vertices \ref{Mor: trivalent}, and each edge is pointing upward. We consider the diagram up to rigid isotopy. It is clear that every face in $s$ has at least 2 edges. Then by Lemma \ref{lem: small face} and Lemma \ref{lem:ngon} we are done. 
\end{proof}

\begin{proof}[Proof of Theorem~\ref{mainthm}]
Denote by $\cP$ the strict pivotal category $\cC(X,F)/\cR$. Then one can trivially define a pivotal functor $P: \cP \to \cD_0$. $P$ is surjective, since all the generators and relations in $\cP$ come from generators and relations in $\cD_0$. Let $f: x \to y $ in $\cP$ be given such that $ P(f) = 0$. The source and target of $f$ are both words in $\{g,\ g^*\}$. If $x \not\simeq y,$ then $f = 0$ since all generating morphisms are isomorphisms. If $x \simeq y,$ then by composing with isomorphisms, $r,s,$ we obtain a morphism $f_1 = r \circ f \circ s  : g \to g$ for some $g$ with $P(f_1)=0$. Since $f_1 = \lambda \id_g$ by Proposition \ref{lem:g_to_g}, $P(f_1)=0$ implies $f_1=0.$  Hence $f=0$ and $P$ is an equivalence.
\end{proof}

\bibliographystyle{alpha}
\bibliography{bibliography.bib}
\end{document}